\documentclass[a4paper, 12pt]{article}

\usepackage[english]{babel}
\usepackage{xcolor}

\usepackage[utf8]{inputenc}
\usepackage[T1]{fontenc}
\usepackage{lmodern}
\usepackage{amsmath}
\usepackage{amsthm}
\usepackage{mathtools}
\usepackage{amssymb}
\usepackage{tikz-cd}
\usepackage{slashed}
\usepackage{parskip}
\usepackage{amsfonts}
\usepackage{thmtools,xcolor}
\usepackage{hyperref}
\usepackage{cleveref}
\usepackage{mathrsfs}
\usepackage{enumitem}
\usepackage{subfiles}
\usepackage{imakeidx}
\usepackage{microtype}
\usepackage[toc,page]{appendix}
\usepackage{numprint}
\usepackage{setspace}
\usepackage{accents}
\usepackage{emptypage}
\allowdisplaybreaks[1]

\usepackage[a4paper]{geometry}
\theoremstyle{plain}
\newtheorem{theorem}{Theorem}[section]
\newtheorem{cor}[theorem]{Corollary}
\newtheorem{lem}[theorem]{Lemma}
\newtheorem{prop}[theorem]{Proposition}

\theoremstyle{definition}

\newtheorem{examps}[theorem]{Examples}
\newtheorem{dfn}[theorem]{Definition}
\newtheorem{rem}[theorem]{Remark}

\theoremstyle{remark}

\newcommand{\RNum}[1]{\uppercase\expandafter{\romannumeral #1\relax}}

\makeatletter
\providecommand*{\twoheadrightarrowfill@}{%
	\arrowfill@\relbar\relbar\twoheadrightarrow
}
\providecommand*{\twoheadleftarrowfill@}{%
	\arrowfill@\twoheadleftarrow\relbar\relbar
}
\providecommand*{\xtwoheadrightarrow}[2][]{%
	\ext@arrow 0579\twoheadrightarrowfill@{#1}{#2}%
}
\providecommand*{\xtwoheadleftarrow}[2][]{%
	\ext@arrow 5097\twoheadleftarrowfill@{#1}{#2}%
}
\makeatother

\newcommand\norm[1]{\left\lVert#1\right\rVert}

\makeatletter
\def\ind{\@ifnextchar[{\@with}{\@without}}
\def\@with[#1]#2{\mathrm{Ind}(#1,#2)}
\def\@without#1{\mathrm{Ind}(#1)}
\makeatother

\newcounter{para}[section]

\newcommand\KK[0]{K\! K}

\DeclareMathOperator\End{End}

\DeclareMathOperator\Ind{Ind}

\DeclareMathOperator\Aut{Aut}

\DeclareMathOperator\Ad{Ad}

\DeclareMathOperator\ad{ad}

\DeclareMathOperator\tr{Tr}
\DeclareMathOperator\ch{Ch}

\DeclareMathOperator\ev{ev}

\DeclareMathOperator\dnc{DNC}

\DeclareMathOperator\Hom{Hom}

\DeclareMathOperator\dom{Dom}

\newcommand{\R}{\mathbb{R}}

\newcommand{\N}{\mathbb{N}}
\newcommand{\C}{\mathbb{C}}
\newcommand{\Z}{\mathbb{Z}}

\begin{document}
\title{Index theorem for inhomogeneous hypoelliptic differential operators}
\author{Omar Mohsen}
\date{}
\maketitle
\abstract{We prove an index theorem for inhomogeneous differential operators satisfying the Rockland condition (hence hypoelliptic). This theorem extends an index theorem for contact manifolds by Van-Erp \cite{MR2680395,MR2680396}.}
\setlength{\parskip}{4pt}
\section*{Introduction}

Let $M$ be a smooth manifold, $H\subseteq TM$ a subbundle. One can define a pseudo-differential calculus where a vector field $X$ has order $1$ if $X\in \Gamma(H)$ and $2$ otherwise. This calculus was first defined in the special case where $X$ is a CR manifold and $H$ the contact structure by Folland and Stein \cite{MR0344699,MR0367477}, in attempts to create a pseudo-differential calculus which contains a parametrix to the Kohn Laplacian and the Laplace operator that appears in Hormander's sum of squares theorem \cite{MR0222474}. Subsequent development by \cite{MR0493005,MR0370271,MR660648,MR973272,MR1185817,MR764508,
MR0423432,MR953082,MR0442149,MR0436223,Erp:2015aa,Dave:aa,MR2417549,MelinUnpubl,MR739894} \footnote{This list is certinely not exahustive.} of this calculus lead to the definition in the general case where $H$ is an arbitrary subbundle. Even more generally if one is given an increasing filtration by subbundles of the tangent bundle $$0=H^0\subseteq H^1\subseteq H^2\cdots\subseteq H^{r-1}\subseteq H^{r}=TM$$ which satisfy the condition $$[\Gamma(H^i),\Gamma(H^j)]\subseteq \Gamma(H^{i+j})\quad\forall i,j,$$then one can define a pseudo-differential calculus where a vector field $X$ has order $i$ if it belongs to $\Gamma(H^i).$

If $x\in M$, then one can define (using the Baker-Campbell-Hausdorff formula) a graded nilpotent Lie group structure on $$\oplus_{i=1}^{r}H^i_{x}/H^{i-1}_x$$ where the Lie bracket is defined by $$[X(x),Y(x)]:=[X,Y](x)\mod H^{i+j-1}_x,\; X\in \Gamma(H^i),Y\in \Gamma(H^j).$$ In the previous definition one can show that $[X,Y](x)\mod H^{i+j-1}_x$ only depends on $X(x)\mod H^{i-1}_x$ and $Y(x)\mod H^{j-1}_x$. Hence the Lie bracket is well defined.

The principal symbol of a differential operator on $M$ in the above calculus at point $x$ is an equivariant differential operator on the Lie group $\oplus_{i=1}^{r}H^i_{x}/H^{i-1}_x$. More generally for pseudo-differential operators, the principal symbol is an (unbounded) multiplier of the $C^*$-algebra $C^*\oplus_{i=1}^{r}H^i_{x}/H^{i-1}_x$.\footnote{For pseudo-differential operators of negative order, the symbol is only a bounded multiplier of $\ker(1_G)$, where $1_G:C^*\oplus_{i=1}^{r}H^i_{x}/H^{i-1}_x\to \C$ is the trivial representation.}

Let us mention two difficulties that appear when one tries to develop this calculus \begin{enumerate}
\item The structure of the groups $\oplus_{i=1}^{r}H^i_{x}/H^{i-1}_x$ is not locally constant in $x$.
\item Even if the structure of the groups is locally constant, the equivalent of Darboux's theorem (the case of $H$ is a contact structure) is not true in general. See \cite{MR1509120}.
\end{enumerate}
If the principal symbol is invertible at everypoint in the above calculus, then the operator, as well as its adjoint, are hypoelliptic. Such an operator will be called $H$-elliptic. A simple criteria equivalent to invertability of the principal symbol was conjectured by Rockland. This was proved in \cite{MR537467,MR499352,MR739894}. \begin{theorem}[Rockland condition]
The principal symbol $\sigma_x(D)$ of a differential operator $D$ at $x\in M$ is left invertible if for every unitary non trivial representation $\pi$ of the group $\oplus_{i=1}^{r}H^i_{x}/H^{i-1}_x$, the (unbounded) operator $\pi(\sigma_x(D))$ is injective.
\end{theorem}

 It follows from hypoellipticity of $D$ and of its formal adjoint $D^t$ and Rellich–Kondrachov theorem that on a compact manifold, an $H$-elliptic operator $D$ is Fredholm. It is quite natural to try to find a topological formula for the analytic index similar to the Atiyah-Singer index formula.

 In the case of contact manifolds, this problem was solved by Van-Erp \cite{MR2680395,MR2680396}. His proof is based on adapting a proof of the Atiyah-Singer index theorem by A. Connes \cite{MR1303779} using tangent groupoids. His solution was published in a $2$-part paper. In part 1, he constructs a deformation groupoid like the tangent groupoid which captures\footnote{See \cite{MR3407514,Erp:2015aa} for more details on the relation between deformation groupoids and pseudo-differential calculus.} the above calculus. Then he proves that solving the index theorem amounts to inverting the Connes-Thom isomorphism for the Heisenberg group. In part 2, he inverts this isomorphism for scalar valued operators. Together with P. Baum \cite{MR3261009}, they improved the formula and the computations done in part 2, so that it applies to the operators with coefficients in a vector bundle.
 
In this article we extend Van-Erp's formula to the general case of an arbitrary filtration. The first step is the construction of the deformation groupoid for the general case of an arbitrary filtration. This was first constructed in \cite{Erp:2016aa,Choi:2015aa}. An elementary construction was given by the author in \cite{Mohsen:2018aa}. The construction given in \cite{Mohsen:2018aa}, shows that solving the index theorem amounts to inverting the Connes-Thom isomorphism like in the contact case.

  Recall that for a bundle of nilpotent Lie groups $G\to M$ on a compact manifold, the Connes-Thom isomorphism (denoted by $\Ind_G$) is an isomorphism $$K(C_0(\mathfrak{g}^*))\xrightarrow{\Ind_G} K(C^*G).$$The principal symbol of an $H$-elliptic operator $D$ defines naturally an element $$[\sigma (D)]\in K(C^*\oplus_{i=1}^{r}H^i/H^{i-1}).$$
  It is shown in \cref{Sec Carnot diff calc}, that $$\dim(\ker(D))-\dim(\ker(D^t))=\Ind_{AS}\left(\Ind^{-1}_G\left([\sigma(D)]\right)\right),$$ where $\Ind_{AS}$ the Atiyah-Singer index map.

A crucial part in Van-Erp's formula is that if $(M,H)$ is a contact manifold with $H$ cooriented, then $H_x\oplus T_xM/H_x=H_x\oplus \R$ is the Heisenberg group for every $x\in M$. Furthermore there is no topological obstruction to finding a compatible almost complex structure. The Bargmann-Fock representation constructed from an almost complex structure gives a continuous field of irreducible representations of the groups $H_x\oplus \R$ for $x\in M$.

In the general case of an arbitrary filtration, it is not possible to construct a continuous field of infinite dimensional irreducible representations because the group bundle is not locally constant. We propose the following solution; let $G_x$ be the group $\oplus_{i=1}^{r}H^i_{x}/H^{i-1}_x$. We can embed the group $G_x$ inside a group $\bar{G}_x$ which lies in a central exact sequence $$0\to \R\to \bar{G}_x\to \mathfrak{g}_x^* \rtimes_{\Ad^*}G_x\to 0.$$
By the construction of $\bar{G}_x$, the left regular representation of $G_x$ extends to an irreducible representation of $\bar{G}_x.$ Also the groups $\bar{G}_x$ form a bundle $\bar{G}=\sqcup_{x\in M}\bar{G}_x$ of graded nilpotent Lie groups on $M$. 

Then starting from an $H$-elliptic operator $D$, we construct a symbol on $\bar{G}$ which satisfies the Rockland condition, and whose class in $K$-theory is equal to the image of $[\sigma(D)]$ under the natural map $$K(C^*G)\to K(C^*\bar{G})$$ which is equal to the composition $$K(C^*G)\xrightarrow{\Ind_G^{-1}} K(C_0(\mathfrak{g}^*))\xrightarrow{\mathrm{Thom}}K(C_0(\bar{\mathfrak{g}}^*))\xrightarrow{\Ind_{\bar{G}}}K(C^*\bar{G})$$ where $\mathrm{Thom}$ is the topological Thom-isomorphism, $\bar{\mathfrak{g}}$ is the Lie algebra of $\bar{G}.$

Let us remark that a cruical part in our construction is that we don't compute the image of $[\sigma(D)]$ in $K(C^*\bar{G})$ by inverting the Connes-Thom isomorphism. Instead, since $\bar{G}$ is equal to $(\mathfrak{g}^*\oplus \R)\rtimes G$, we can guess the image by writing the formula as if $\bar{G}$ was equal to $(\mathfrak{g}^*\oplus \R)\times G$. Then in \cref{main thm on ind crossed product}, we show that it is in fact the image  by creating a homotopy between $\bar{G}$ and the direct product $(\mathfrak{g}^*\oplus \R)\times G$.

The group bundle $\bar{G}$ shares the following property with the Heisenberg group : its $C^*$-algebra $C^*\bar{G}$ lies in an exact sequence \begin{equation}\label{eqn intro exact}
 0\to C_0(\R^*)\otimes \mathcal{K}(L^2G)\to C^*\bar{G}\to C^*\left(\mathfrak{g}^*\rtimes_{\mathrm{Ad^*}} G\right)\to 0.
\end{equation} A similar exact sequence holds for bundles of Heisenberg groups. In \cite{MR2680396}, Van-Erp uses the above exact sequence for the Heisenberg group together with a trick to invert the Connes-Thom isomorpphism. In fact one easily shows that in the $6$-term exact sequence the map $$K(C_0(\R^*)\otimes \mathcal{K}(L^2G))\to K(C^*\bar{G})$$ is onto. Van-Erp's trick is to use a natural anti group automorphism of the Heisenberg group to create a per-image. This trick doesn't work in general for $\bar{G}$ because the anti group automorphism that Van-Erp uses is no longer an anti group automorphism for $\bar{G}$. Nevertheless a variant of his trick by instead using the group automorphism $(\xi,t)\to (\xi,-t)$ on the $\mathfrak{g}^*\times \R$ part of $\bar{G}$ does work for the image of $[\sigma(D)]$ in $K(C^*\bar{G})$. A topological formula is then obtained in \cref{index formula thm}.

The article is organised as follows 
\begin{enumerate}
\item In \Cref{Rockl sect}, we recall the Rockland condition for differential operators on Carnot groups. We then give some examples.
\item In \Cref{Connes-Thom sect}, we recall the construction of Connes-Thom isomorphism. We then give the construction of push-forward of differential operators satisfying the Rockland condition.
\item In \Cref{sect Index thm}, we invert the Connes-Thom isomorphism and obtain an index theorem for inhomogeneous differential operators (at the level of symbols).
\item In \Cref{Sec Carnot diff calc}, we recall some facts on the inhomogeneous pseudo-differential calculus on a smooth manifold. In this section we prove that inverting Connes-Thom isomorphism gives the index theorem. Which together with the index theorem in \cref{sect Index thm} gives the desired index formula.
\item In the appendix \ref{Sect Symbols}, we added all the analytic results needed. We recall the definition of symbols of pseudo-differential operators. This is needed to justify some results on the inhomogeneous differential operators.
\end{enumerate}
\section*{Acknowledgement}
I would like to thank Echterhoff and Skandalis for their suggestions. This work was done while the author was a postdoc supported by the excellence cluster in university of Münster.
\section*{Notation}
\begin{itemize}
 \item We use $\N=\{1,2,\cdots,\},$ $\N_0=\N\cup\{0\}.$
 \item  If $V$ is a vector bundle, then we will denote by $\Gamma(V)$ the space of smooth sections of $V$. 
 \item We will freely use the theory of Hilbert modules, and $C(X)$-$C^*$-algebras. We refer the reader to \cite{MR1325694,MR918241,MR1686846} for the necessary definitions.
\item We will write $K(A)$ to mean the $K$ theory of a $C^*$-algebra $A$, when the degree ($K_0$ or $K_1$) is irrelevant or clear from the discussion.
\end{itemize}
\section{Rockland condition on Carnot groups}\label{Rockl sect}
\paragraph{Carnot groups.}Let $G$ be a Carnot Lie group.\footnote{The terminology is not the same across literature.} This means that $G$ is a simply connected Lie group and its lie algebra is equipped with a decomposition $\mathfrak{g}=\oplus_{i=1}^n\mathfrak{g}_i$ such that $[\mathfrak{g}_i,\mathfrak{g}_j]\subseteq \mathfrak{g}_{i+j}$ for all $i,j\in \N$ with the convention, $\mathfrak{g}_k=0$ if $k>n$.  The group is sometimes called $n$-step Carnot group to emphasize the number of components. The Baker-Campbell-Hausdorff formula allows us to identify $G$  with $\mathfrak{g},$ which we will do without further mention.

Let $V_0, V_1$ be Hermitian vector spaces. We will identify the enveloping algebra $$\mathcal{U}(\mathfrak{g},V_0,V_1):=\mathcal{U}(\mathfrak{g})\otimes_\C \Hom(V_0,V_1)$$ with the algebra of right $G$-invariant differential operators on $G$ with coefficients in $\Hom(V_0,V_1)$. If $V_0=V_1$, we will simply write $\mathcal{U}(\mathfrak{g},V_0).$

For $\lambda\in \R$, let $\alpha_\lambda\in \End(\mathfrak{g})$ be the Lie algebra endomorphism $$\alpha_\lambda(x)=\lambda^i x,\quad x\in \mathfrak{g}_i.$$The Baker-Campbell-Hausdorff formula shows that they are also Lie group homomorphisms. We extend them to $\mathcal{U}(\mathfrak{g},V_0,V_1)$ by acting trivially on $\Hom(V_0,V_1)$. We will denote by $\mathcal{U}_k(\mathfrak{g},V_0,V_{1})$ the set of elements $D\in \mathcal{U}(\mathfrak{g},V_0,V_1)$ such that $\alpha_\lambda (D)=\lambda^kD$ for $\lambda\in \R^*.$ One has $\mathcal{U}(\mathfrak{g},V_0,V_1)=\oplus_{k\geq 0}\mathcal{U}_k(\mathfrak{g},V_0,V_1).$
 
We will always assume that a Carnot Lie group is equipped with a left invariant Riemannian metric such that the different $\mathfrak{g}_i's$ are orthogonal to each other. This fixes the choice of a Haar measure. 
\begin{prop}\label{div of X}Let $G$ be a Carnot Lie group. If $X\in \mathfrak{g}$, then $\mathrm{div}(X)=0$.\end{prop}
\begin{proof}
On a Lie group it is well know that $\mathrm{div}(X)=-\tr(\ad_X)$. This formula follows from Cartan's formula for the exterior derivative of differential forms. The trace is $0$ for a nilpotent Lie group.
\end{proof}
It follows from \cref{div of X} that the formal adjoint of a differential operator in $\mathcal{U}_k(\mathfrak{g},V_0,V_1)$ is in $\mathcal{U}_k(\mathfrak{g},V_1,V_0)$.

In this article we are interested in \textit{bundles} of Carnot Lie groups. A bundle of ($n$-step) Carnot Lie groups over a compact manifold $M$ is an $\N$-graded vector bundle $\mathfrak{g}=\oplus_{i=1}^n\mathfrak{g}_i\to M$ equipped with a smooth family of Lie brackets $[\cdot,\cdot]$ such that $[\mathfrak{g}_i,\mathfrak{g}_j]\subseteq \mathfrak{g}_{i+j}.$ The Baker-Campbell-Hausdorff formula implies that the fiberwise product map $G\times_M G\to G$ is smooth. We don't put any conditions of local triviality of the group bundle. 

We will always assume that the bundle $\mathfrak{g}$ is equipped with a Euclidean metric such that different $\mathfrak{g}_i$'s are orthogonal to each other. Let $V_0,V_1\to M$ be Hermitian vector bundles, $D\in \Gamma(\mathcal{U}_k(\mathfrak{g},V_0,V_1))$. Since $D$ is right invariant, it follows that for $f,g\in C^\infty_c(G)$, $$D(f\star g)=D(f)\star g.$$ We will denote by $$\overline{D}:\dom(\overline{D})\subseteq \Gamma(V_0)\otimes_{C(M)} C^*G\to \Gamma(V_1)\otimes_{C(M)} C^*G$$ the closure of $D$. The $C^*$-adjoint of $\overline{D}$ contains $\overline{D^t}$, where $D^t$ denotes the formal adjoint of $D$. It follows that $\overline{D}$ has a densely defined adjoint.

\paragraph{The Rockland condition.}
Let $x\in M$, $\pi:G_x\to \mathcal{U}(B(H))$ be a unitary representation. It follows that $\pi$ can be regarded as a representation of $C^*G$. By \cite[lemma 1.15]{MR2270655}, $$\pi(\overline{D}):\dom(\pi(\overline{D}))\subseteq V_{0,x}\otimes H\to  V_{1,x}\otimes H$$ is a well defined closed unbounded operator.

\begin{dfn}An operator $D\in\Gamma(\mathcal{U}_k(\mathfrak{g},V_0,V_1))$ is said to satisfy the Rockland condition at a point $x$ if for every non-trivial irreducible representation $\pi$ of $G_x$, $\pi(\overline{D})$ and $\pi(\overline{D^t})$ are injective. The operator $D$ is said to satisfy the Rockland condition if it does at every point $x\in M.$
\end{dfn}
If $G$ is commutative, then an operator $D\in \Gamma(\mathcal{U}_k(\mathfrak{g},V_0,V_1))$ under the Euclidean Fourier transform becomes a polynomial $\sigma(D):\mathfrak{g}^*\to \End(V_0,V_1)$. The Rockland condition at $x$ becomes the usual (graded) ellipticity condition that $\sigma(D)(x,\xi)$ is an isomorphism for every $\xi\in \mathfrak{g}_x^*\backslash \{0\}.$

\begin{prop}\label{Rockland open}Let $D\in \Gamma(\mathcal{U}_k(\mathfrak{g},V_0,V_1))$ be an operator. The set of all $x\in M$ such that $D$ satisfy the Rockland condition at $x$ is an open set.
\end{prop}
\begin{prop}\label{Rockland equiv}
Let $D\in \Gamma(\mathcal{U}_k(\mathfrak{g},V_0,V_1))$ for some $k\geq 0$ be an operator that satisfies the Rockland condition, then
\begin{enumerate}
\item The operator $\overline{D}$ is regular in the sense of Baaj-Julg \cite{MR715325}, and $\overline{D}^*=\overline{D^t}.$
\item the operator $$\bar{D}(1+\bar{D}^*\bar{D})^{-\frac{1}{2}}\in \mathcal{L}(\Gamma(V_0)\otimes_{C(M)} C^*G,\Gamma(V_1)\otimes_{C(M)} C^*G)$$ is a Fredholm operator in the sense of Kasparov $C^*$-modules.
\end{enumerate}
\end{prop}
The proof of propositions \ref{Rockland open}, \ref{Rockland equiv} is given in \cref{Sect Symbols}.
\begin{rem}I don't know if the operator $\overline{D}$ is regular without the Rockland condition. This is (trivially) true for commutative groups. It makes no difference in what follows.
\end{rem}

Let $D\in \Gamma(\mathcal{U}_k(\mathfrak{g},V_0,V_1))$ satisfies the Rockland condition, then the element \begin{align*}
[D]:=[\Gamma(V_0\oplus V_1)\otimes_{C(M)}C^*G,\begin{pmatrix}
0 & \overline{D}^*(1+\overline{D}\overline{D}^*)^{-\frac{1}{2}} \\  \overline{D}(1+\overline{D}^*\overline{D})^{-\frac{1}{2}} &0
\end{pmatrix}]\\\in \KK_0(\C,C^*G)=K_0(C^*G).
\end{align*} In the ungraded case, a bundle $D\in \Gamma(\mathcal{U}_k(\mathfrak{g},V))$ of symmetric differential operators satisfying the Rockland condition gives an element $$[D]:=[(\Gamma(V)\otimes_{C(M)}C^*G),\overline{D}(1+\overline{D}^*\overline{D})^{-\frac{1}{2}}]\in \KK_1(\C,C^*G)=K_1(C^*G).$$
\paragraph{Functoriality}Let $G_1,G_2$ be bundles of Carnot Lie groups over $M$, $\phi:G_1\to  G_2$ a bundle of Carnot Lie group homomorphisms. Here Carnot homomorphism means that $\phi(\mathfrak{g}_{1,i})\subseteq \mathfrak{g}_{2,i}$ for every $i$. One naturally gets a map $\phi:\mathcal{U}(\mathfrak{g}_1,V)\to \mathcal{U}(\mathfrak{g}_2,V).$

If $\phi$ is surjective, then integrating along the fibers gives a surjective $C^*$-morphism $$C^*\phi:C^*G_1\to C^*G_2.$$ In that case if $D\in\mathcal{U}(\mathfrak{g}_1,V) $ satisfies the Rockland condition, then $\phi(D)$ does as well and \begin{equation}\label{eqn funct k theory class}
 K(C^*\phi)([D])=[\phi(D)]\in K(C^*G_2).
\end{equation}

Let us remark that the last identity is quite silly because \cref{cor connes-Thom} says that $K(C^*\phi)$ is the zero map whenever $\dim(G_1)>\dim(G_2)$. On the other hand its contrapositive can be quite useful : if $D\in  \mathcal{U}(\mathfrak{g},V)$ is an operator, $\phi$ a non-injective surjective Carnot homomorphism such that $\phi(D)$ satisfies the Rockland condition and $[\phi(D)]\neq 0$, then $D$ doesn't satisfy the Rockland condition.
\begin{examps}\label{examps sect 1}\begin{enumerate}
\item Let $s\in \N$ be such that $i$ divides $s$ for every $i=1,\cdots,n$. If $X_i\in \Gamma(\mathfrak{g})$ is an orthonormal homogeneous basis of $\mathfrak{g}.$ Here homogeneous means that each $X_i\in\Gamma(\mathfrak{g}_{w(i)})$ for some $w(i)\in \N$ , then the operator $$\sum (-1)^{\frac{s}{w(i)}}X_i^{\frac{2s}{w(i)}}\in \Gamma(\mathcal{U}_{2s}(\mathfrak{g}))$$ is a positive operator that satisfies the Rockland condition.
\item Let $V\to M$ be a Hermitian vector bundle, $c:\mathfrak{g}\to \End(V)$ a bundle of linear maps such that \begin{align*}
c(X)^2=\norm{X}^2\mathrm{Id}_V,\quad c(X)^*=c(X),\quad \forall X\in \mathfrak{g}.
\end{align*}
For each $i$, let $$D_i=\sqrt{-1}\sum_{j}c(X_j)X_j,$$ where $X_j$ is an orthonormal basis of $\mathfrak{g}_i$. Consider $s\in ]n,+\infty[$. The following symmetric operator\footnote{It is only a pseudo-differential operator, see \cref{Sect Symbols}.} satisfies the Rockland condition $$D=\sum_{i=1}^{n}|D_i|^{\frac{s}{i}-1}D_i\in \Gamma(\mathcal{U}_s(\mathfrak{g},V)).$$ To see this let $x\in M$, $\pi$ be any non-trivial irreducible representation of $G_x$ acting on a Hilbert space $H$. By Kirillov's theory \cite{MR0133406}, $\pi$ corresponds to a nonzero element $l=(l_1,\cdots,l_n)\in \oplus_{i=1}^n\mathfrak{g}_{i,x}^*$. Let $k$ be the maximal $i$ such that $l_i\neq 0$. For every $j>k$, $\pi(\mathfrak{g}_{j,x})=0$, and $\pi(\mathfrak{g}_{k,x})=\C\cdot \mathrm{Id}.$ It follows that acting on $H\otimes V_x$ $$\pi(D)=\pi\left(\sum_{i<k}|D_i|^{\frac{s}{i}-1}D_i\right)+\norm{l_k}^{\frac{s}{k}-1}\mathrm{Id}\otimes  c(\#l_k),$$ where $\#:\mathfrak{g}_k^*\to\mathfrak{g}_k$ is the isomorphism given by the Euclidean metric. The first operator is symmetric and anti-commutes with the second. It follows that $\pi(D)^2$ is the sum of a positive operator and $\norm{l_k}^{\frac{2s}{k}}\cdot \mathrm{Id}$. Hence $\pi(D)$ is injective.

\item Suppose that for each $i\geq 1$, $\mathfrak{g}_{i+1}=[\mathfrak{g}_1,\mathfrak{g}_{i}]$. If $X_k$ denotes a base of $\mathfrak{g}_1$, then the following operator satisfies the Rockland condition$$D=\sum_{k}X_k^2.$$

One might also consider a Dirac operator $$D=\sum_{k}c(X_k)X_k,$$ where $c:\mathfrak{g}_1\to \End(V)$ is a Clifford action. This operator \textit{doesn't} satisfy the Rockland condition in general. Let $\phi:G\to \mathfrak{g}_1$ be the projection map onto $\mathfrak{g}_1$ with the trivial Lie bracket. If $\mathfrak{g}_1\subsetneq \mathfrak{g}$, then a necessary condition for $D$ to satisfy Rockland condition is that the class of $[\phi(D)]\in K(C_0(\mathfrak{g}_1^*))$ is $0$. This is the classical Dirac-type operator whose class is usually nonzero. Such operators for contact manifolds were studied in \cite{MR2122745,MR3403787,ThesisDiracContact}. In particular in \cite{ThesisDiracContact}, a way to modify the operator so that it satisfies the Rockland condition is given.
\item Let $M=\{pt\}$, $G$ be a $2$-step Carnot group with $\mathfrak{g}_2=\R T$, $X_i$ an orthonormal basis of $\mathfrak{g}_1$, $\gamma\in M_n(\C)$. Consider the operator $$D=-\sum_{i}X_i^2\mathrm{Id}_{\C^n}+ \sqrt{-1}\gamma T\in \mathcal{U}_2(\mathfrak{g},\C^n,\C^n).$$

Let $2n$ be the rank of $[\cdot,\cdot]:\mathfrak{g}_1\times \mathfrak{g}_1\to \R$, $\pm \sqrt{-1}\lambda_1,\cdots,\pm \sqrt{-1}\lambda_{n}$ with $\lambda_k>0$ the nonzero eigenvalues of $[\cdot,\cdot]$ counted with multiplicity.

It is proved in \cite[section 3.4]{MR2417549} that $D$ satisfies the Rockland condition if and only if the spectrum of $\gamma$ avoids the singular set \begin{align*}
]-\infty,-\sum_{k=1}^n \lambda_k]\cup [\sum_{k=1}^n \lambda_k,+\infty[\quad\text{if}\; 2n<\dim(\mathfrak{g}_1)\\
\{\pm (\sum_{k=1}^n \lambda_k+2\sum_{k=1}^n\alpha_j\lambda_j):\alpha\in \N_0^n\}\quad\text{if}\; 2n=\dim(\mathfrak{g}_1)
\end{align*}
\end{enumerate}
\end{examps}
\section{Connes-Thom isomorphism}\label{Connes-Thom sect}
\paragraph{Connes-Thom isomorphism}
Let $G$ be a bundle of Carnot Lie groups over a compact manifold $M$. One can define a bundle of Carnot Lie groups over $M\times [0,1]$, denoted by $\dnc(G,M)$.\footnote{See \cite{ClaireGeorges} for the general construction of the deformation to the normal cone and its relation to Lie groupoids. The construction given here is more adapted to Carnot groups and is slightly different from that of \cite{ClaireGeorges}} The bundle $\dnc(G,M)$ is defined as follows; over $(x,t)\in M\times [0,1]$, its Lie algebra is the Lie algebra $\mathfrak{g}_x$ with the Lie bracket $$[g,h]:=\alpha_t([g,h]_{\mathfrak{g}_{x}}),$$ where $\alpha$ are the endomorphisms defined in \cref{Rockl sect}. It is clear that $$\dnc(G,M)=G\times ]0,1]\sqcup \mathfrak{g}\times \{0\}.$$
 Its $C^*$-algebra $C^*(\dnc(G,M))$ lies in the exact sequence $$0\to C_0(]0,1])\otimes C^*G\to C^*(\dnc(G,M))\xrightarrow{\ev_0}  C_0(\mathfrak{g}^*)\to 0. $$
Since the $C^*$-algebra $C_0(]0,1])$ is contractible, it follows that the map $\ev_0$ is an isomorphism in $K$-theory. The composition $$K_*\left(C_0(\mathfrak{g}^*)\right)\xrightarrow{K(\ev_0)^{-1}} K_*(C^*(\dnc(G,X)))\xrightarrow{K(\ev_1)}K_*(C^*G)$$ is called the Connes-Thom isomorphism and will be denoted by $\Ind_G$
 
\begin{rem}\label{def index map}Later on different deformations will appear, to each one we will have an index map in $K$-theory constructed exactly as above.
\end{rem}
\begin{theorem}[Connes \cite{MR605351}, see also \cite{MR600000}]\label{Connes-Thom}The map $\Ind_G$ is an isomorphism.
\end{theorem}
\begin{cor}\label{cor connes-Thom}Let $G_1,G_2\to M$ be bundles of Carnot groups, $\phi:G_1\to G_2 $ a bundle of surjective Carnot Lie group homomorphisms. If $\dim(G_1)>\dim(G_2)$, then $$K(C^*\phi):K(C^*G_1)\to K(C^*G_2)$$ is the zero map.
\end{cor}
\begin{proof}
One naturally has a bundle of surjective Carnot homomorphisms $\dnc(\phi):\dnc(G_1,M)\to \dnc(G_2,M).$ It follows from the definition of the index map that the following diagram commutes $$\begin{tikzcd}K(C^*G_1)\arrow[r,"K(C^*\phi)"]&K(C^*G_2)\\K(C_0(\mathfrak{g}^*_1))\arrow[u,"\Ind_{G_1}"]\arrow[r,"K(\phi^*)"]&K(C_0(\mathfrak{g}_2^*))\arrow[u,"\Ind_{G_2}"]\end{tikzcd}.$$
The lower map is the pullback map $\phi^*:\mathfrak{g}^*_2\to \mathfrak{g}_1^*$. It is clear that $K(\phi^*)$ is the zero map. The result follows from \cref{Connes-Thom}.\end{proof}

\paragraph{Dirac Operator}Let $c:\mathfrak{g}\to \End(V)$ be a spin$^c$ structure on $\mathfrak{g}.$ In \cref{examps sect 1}, we constructed an operator $D$ associated to $c$ which satisfies the Rockland condition. Since one could do the same construction on the bundle $\dnc(G,M)\to  M\times [0,1]$. It follows from the definition of the index map $$\Ind_G^{-1}([D])=[c]$$ where $[c]$ is the Thom class associated to $c$. 
\paragraph{Crossed product}The following construction appears when one calculates products in $K$-theory \cite{MR0236950}. If $E_0, E_1,F_0,F_1$ are Hilbert spaces, $\alpha:E_0\to E_1,\;\beta:F_0\to F_1$ are closed (unbounded) operators, then one defines the operator $$\alpha\# \beta=\begin{pmatrix}
\alpha\otimes Id_{F_0} & -Id_{E_1}\otimes \beta^*\\Id_{E_0}\otimes \beta& \alpha^*\otimes Id_{F_1}
\end{pmatrix}:(E\otimes F)_0\to (E\otimes F)_1,$$ where $(E\otimes F)_0=E_0\otimes F_0\oplus E_1\otimes F_1$ and  $(E\otimes F)_1=E_1\otimes F_0\oplus E_0\otimes F_1$.

We imitate this construction to crossed products. Let $E,F\to M$ two graded Hermitian vector bundles, $G,N,H$ bundles of Carnot Lie groups over $M$. We will suppose that $N,H$ are Carnot Lie subgroups\footnote{Lie subgroups such that the inclusion is a Carnot homomoprhism} of $G$ and $G=N\rtimes H$.

Let $D_1\in \Gamma(\mathcal{U}_k(\mathfrak{n},E_0,E_1)), D_2\in \Gamma(\mathcal{U}_k(\mathfrak{h},F_0,F_1))$ be operators satisfying the Rockland condition with $k\in \N$. Let $i_N,i_H$ be the inclusions of $N$ and $H$ inside $G$. We consider the operator $$D_1\# D_2=\begin{pmatrix}i_N(D_1)\otimes Id_{F_0} & Id_{E_1}\otimes-i_H(D_2^t)\\Id_{E_0}\otimes i_H(D_2) & i_N(D_1^t)\otimes Id_{F_1}\end{pmatrix}\in\mathcal{U}_k(\mathfrak{g},(E\otimes F)_0,(E\otimes F)_1),$$where $(E\otimes F)_0=E_0\otimes F_0\oplus E_1\otimes F_1$ and  $(E\otimes F)_1=E_1\otimes F_0\oplus E_0\otimes F_1$.
\begin{theorem}\label{main thm on ind crossed product}
\begin{enumerate}
\item There exists $\epsilon>0$ such that for $c\in ]0,\epsilon[$, the operator $D_1\# cD_2$ satisfies the Rockland condition.
\item For any $c\in ]0,\epsilon[$, $$\Ind_{G}^{-1}([D_1\#c D_2])=\Ind_{N}^{-1}([D_1])\otimes \Ind_H^{-1}([D_2])\in K_0(C_0(\mathfrak{g}^*))$$
\end{enumerate} 
\end{theorem}
\begin{rem}
  The theorem remains true if one replaces $D_1$ or $D_2$ by symbols of positive order satisfying the Rockland condition. The same proof holds in this case.
\end{rem}
\begin{proof}
If the crossed product $G=N\rtimes H=N\times H$ was trivial, then the operators $i_N(D_1)\otimes Id_{F_0}$ and $Id_{E_0}\otimes i_H(D_2)$ commute with each other. Since irreducible representations of $N\times H$ are of the form $\pi\otimes \pi'$ for an irreducible representation $\pi$ of $N$ and $\pi'$ of $H$, it follows that $D_1\# D_2$ satisfies the Rockland condition. It is proven in \cite{MR715325}, that $D_1\# D_2$  is the Kasparov product $[D_1]\otimes [D_2]\in K_0(C^*G)$.

For the general case, let $\beta:H\to \Aut(N)$ be the action of $H$ on $N$, $\mathbb{G}=N\rtimes_{\mathbb{\beta}'} H\to M\times [0,1]$ the bundle of Carnot groups where on $M\times \{t\}$, $\beta'=\beta\circ\alpha_t.$ Clearly for $t\neq 0$, the restriction of $\mathbb{G}$ to $M\times \{t\}$ is isomorphic to $G$, and on $M\times \{0\}$, $\mathbb{G}=N\times H$. One defines the operator $D_1\#D_2$ on $\mathbb{G}$. We will denote this operator by $\mathbb{D}$. The operator $\mathbb{D}$ satisfies the Rockland condition at every point in $M\times \{0\}$. It follows from \cref{Rockland open}, that there exists $\epsilon>0$ such that $\mathbb{D}$ satisfies the Rockland condition on $M\times [0,\epsilon].$ Part 1 follows from the isomorphism between $\mathbb{G}_{|M\times\{t\}}$ and $G$ for $t\neq 0.$ For simplicity, we will assume that $\epsilon=c=1$ in the rest of the proof.

The $C^*$-algebra $C^*\mathbb{G}$ lies in an exact sequence $$0\to C^*G\otimes C_0(]0,1])\to C^*\mathbb{G}\to C^*N\otimes C^*H\to 0.$$
By \cref{def index map}, one has an index map $$\Ind_{N\rtimes H}^{N\times H}:K_0(C^*N\otimes C^*H)\to K_0(C^*G).$$Since the operator $\mathbb{D}$ satisfies the Rockland condition on $M\times [0,1]$, it follows from the definition of the index map that $$\Ind_{N\rtimes H}^{N\times H}([D_1]\otimes [D_2])=[D_1\# D_2].$$ The theorem then follows from \cref{commt diagram groups}\qedhere \begin{lem}\label{commt diagram groups}The following diagram commutes $$\begin{tikzcd}[column sep = large]K_0(C_0(\mathfrak{g}^*))\ar[equal]{d}\arrow[r,"\Ind_G"]&K_0(C^*G)\\K_0(C_0(\mathfrak{n}^*\oplus \mathfrak{h}^*))\arrow[r,"\Ind_N\times \Ind_H"]&K_0(C^*N\otimes C^* H)\arrow[u,"\Ind_{N\rtimes H}^{N\times H}"]\end{tikzcd}$$
\end{lem}
\begin{proof}
One has a bundle of Carnot Lie groups over $M\times [0,1]^2$ given by $\dnc(\mathbb{G},M\times [0,1])$. One has evaluation maps $ev_{L}$ from $C^*\dnc(\mathbb{G},M\times [0,1])$ to the restriction of $C^*\dnc(\mathbb{G},M\times [0,1])$ to any closed $L\subseteq [0,1]^2,$ hence an exact sequence $$0\to \ker(\ev_{[0,1]\times\{0\}})\to \ker(\ev_{(0,0)})\xrightarrow{\ev_{[0,1]\times\{0\}}} C_0(\mathfrak{g}^*)\otimes C_0(]0,1])\to 0.$$
The $C^*$-algebra $\ker( \ev_{[0,1]\times\{0\}})$ is isomorphic to $C^*\mathbb{G}\otimes C_0(]0,1])$. Hence $\ker(\ev_{[0,1]\times\{0\}})$ and $C_0(\mathfrak{g}^*)\otimes C_0(]0,1])$ are contractible. Therefore by the $6$ term exact sequence $K_*(\ker(\ev_{(0,0)}))=0$. It follows that $$K(\ev_{(0,0)}):K_0(C^*\dnc(\mathbb{G},M\times [0,1]))\to K_0(C_0(\mathfrak{g}^*))$$ is an isomorphism. If $x\in K_0(C^*\dnc(\mathbb{G},M\times [0,1]))$ then the two compositions of the diagram send $K(\ev_{(0,0)})(x)$ to $K(\ev_{(1,1)})(x).$
\end{proof}
\end{proof}
\begin{rem}If $E_0=E_1$ and $D_1$ is symmetric, the operator $$D_1\# D_2:=\begin{pmatrix}i_N(D_1)\otimes Id_{F_0} & Id_{E}\otimes i_H(D_2^t)\\Id_{E}\otimes i_H(D_2) & -i_N(D_1)\otimes Id_{F_1}\end{pmatrix}\in\mathcal{U}_k(\mathfrak{g},E\otimes F)$$ is symmetric. \Cref{main thm on ind crossed product} still holds in this case with an equality at the level of $K_1(C^*G)$. The proof is the same.
\end{rem}

\section{Inverting Connes-Thom isomorphism }\label{index theorem}\label{sect Index thm}
\subsection{The main Construction}
\paragraph{The dual of a Carnot group.} Let $G$ be an $n$-step Carnot group. Kirillov's theory shows that for a Carnot group $\hat{G}=\mathfrak{g}^*/G.$ Since $\mathfrak{g}_n\subseteq \mathcal{Z}(\mathfrak{g})$, the co-adjoint action of $G$ on $\mathfrak{g}^*=\oplus_{i=1}^n\mathfrak{g}_i^*$ is trivial on the last component.

\begin{prop}\label{non deg group prop}If $\mathfrak{g}_n=\R$ and for each $i\in \{1,\cdots,n-1\}$, the map $$[\cdot,\cdot]:\mathfrak{g}_i\times \mathfrak{g}_{n-i}\to \mathfrak{g}_n=\R$$ is nondegenerate, then $$\hat{G}=\R^*\sqcup \left(\oplus_{i=1}^{n-1}\mathfrak{g}^*_i\right)/G$$
\end{prop}
\begin{proof}
As remarked before the proposition, it is enough to prove that for every $t\neq 0$, an element $(l_1,\cdots,l_{n-1},t)\in \mathfrak{g}^*$ contains $(0,\cdots,0,t)$ in its $G$-orbit. It follows from the Baker-Campbell-Hausdorff formula that if $y\in \mathfrak{g}_k$ for $k\in  \{1,\cdots,n-1\}$, then $$\mathrm{Ad}_y(l_1,\cdots,l_{n-1},t)=(l_1',\cdots,l_{n-k-1}',l_{n-k}+t[y,\cdot],l_{n-k+1},\cdots,l_{n-1},t)$$ for some $l_i'\in \mathfrak{g}_i^*.$ The proposition is then clear.
\end{proof}
\paragraph{The construction of $\bar{G}$.} Let $G$ be an $n$-step Carnot group. The crossed product $\mathfrak{g}^*\rtimes_{\Ad^*} G$ is an $n$-step Carnot group when given the grading $$\oplus_{i=1}^n\mathfrak{g}_i\oplus \mathfrak{g}^*_{n-i+1}.$$ We will define a central extension of the group $\mathfrak{g}^*\rtimes_{\Ad^*} G$. Let $\bar{G}$ be the $(n+1)$-Carnot group whose Lie algebra is equal to $\bar{\mathfrak{g}}=\left(\oplus_{i=1}^n\mathfrak{g}_i\oplus \mathfrak{g}^*_{n-i+1}\right)\oplus \R Z,$where $Z$ is an element of degree $n+1$. The Lie bracket is defined by the formulas \begin{align*}
&[Z,x]_{\bar{\mathfrak{g}}}=0,&\forall x\in \bar{\mathfrak{g}}\\
&[x,y]_{\bar{\mathfrak{g}}}=[x,y]_{\mathfrak{g}},&\forall x\in \mathfrak{g},\;y\in\mathfrak{g}\\
&[x,y]_{\bar{\mathfrak{g}}}=0&\forall x\in \mathfrak{g}^*,\; y\in \mathfrak{g}^*\\
&[x,y]_{\bar{\mathfrak{g}}}=0& \forall x\in \mathfrak{g}_{i}^*,\;y\in \mathfrak{g}_{j},\; i<j\\
&[x,y]_{\bar{\mathfrak{g}}}=j\langle x,y\rangle Z,&\forall x\in \mathfrak{g}_{j}^*,\;y\in \mathfrak{g}_{j}\\
&\langle [x,y]_{\bar{\mathfrak{g}}},z\rangle=\langle x,[y,z]_\mathfrak{g}\rangle&\forall x\in \mathfrak{g}_i^*,y\in \mathfrak{g}_j,\;z\in \mathfrak{g}_{i-j},\; i>j
\end{align*}where $[\cdot,\cdot]_{\mathfrak{g}}$ is the Lie bracket on $\mathfrak{g}$ and $i,j\in \{1,\cdots,n\}$. The Lie bracket $[y,x]:=-[x,y]$ in all the above cases.
\begin{prop}The space $\bar{\mathfrak{g}}$ is the Lie algebra of a Carnot Lie group.
\end{prop}
\begin{proof}
Antisymmetry  and the identity $[\bar{\mathfrak{g}}_i,\bar{\mathfrak{g}}_j]\subseteq \bar{\mathfrak{g}}_{i+j}$ are part of the definition. So the only thing that one needs to check is the Jacobi identity $$[[x,y],z]+[[y,z],x]+[[z,x],y]=0.$$ Since $Z\in \mathcal{Z}(\bar{\mathfrak{g}})$, we will suppose that $x,y,z\in \oplus_{i=1}^{n}\bar{\mathfrak{g}}_{i}.$ We will divide this by linearity into cases\begin{itemize}
\item if $x\in \mathfrak{g}_i, y\in\mathfrak{g}_j,z\in\mathfrak{g}_k$, then the Jacobi identity follows from the same identity on $\mathfrak{g}$.
\item if $x\in \mathfrak{g}^*_i, y\in\mathfrak{g}_j,z\in\mathfrak{g}_k$, then if $j+k>i$, then each term is zero. If $j+k=i$, then one has  \begin{align*}
[[x,y],z]+[[y,z],x]+[[z,x],y]&=k\langle x,[y,z]\rangle-i\langle x,[y,z]\rangle+j\langle x,[y,z]\rangle\\&=0.
\end{align*}
If $i>j+k$, then for $w\in \mathfrak{g}_{i-j-k}$, one has \begin{align*}
&\langle [[x,y],z],w\rangle+\langle [[y,z],x],w\rangle+\langle [[z,x],y],w\rangle=\\&\langle x,[y,[z,w]]\rangle +\langle x,[w,[y,z]]\rangle+\langle x,[z,[w,y]]\rangle=0
\end{align*}
\item if $x\in \mathfrak{g}^*_i, y\in\mathfrak{g}^*_j,z\in\mathfrak{g}_k$, then each of the terms is zero.
\item if $x\in \mathfrak{g}^*_i, y\in\mathfrak{g}^*_j,z\in\mathfrak{g}^*_k$, then each of the terms is zero.\qedhere
\end{itemize}
\end{proof}
\begin{rem}Let $N\subseteq \bar{G}$ be the normal abelian subgroup whose Lie algebra is $\mathfrak{g}^*\oplus \R Z$. It is clear that $\bar{G}=N\rtimes G.$
\end{rem}

\subsection{Index theorem, $K_0$ statement}\label{index sect 3} Let $G\to M$ be a bundle of Carnot Lie groups. It is clear that the construction of $\bar{G_x}$ for each $x\in M$ extends fiber by fiber to yield a bundle of Carnot Lie groups over $M$ denoted by $\bar{G}$.

For $(0,\cdots,0,1)\in \bar{\mathfrak{g}}^*$, it is clear that $\mathfrak{g}^*\oplus \R Z$ is a polarising Lie algebra. See \cite{MR0133406,MR1070979} for the definition of polarising Lie subalgebras. Furthermore a natural choice of a complementary Lie subalgebra is given by $\mathfrak{g}.$ The same holds for $(0,\cdots,0,-1)\in\bar{\mathfrak{g}}^* .$ By Kirillov theory, one has a bundle of irreducible unitary representations $\pi^{\pm}_x:\bar{G_x}\to \mathcal{U}(L^2\mathfrak{g})$ associated to $(0,\cdots,0,\pm 1)$.

 By \cref{non deg group prop}, one has the exact sequence \begin{equation}\label{exactseq Gbar}
   0\to C_0(\R^*)\otimes \mathcal{K}(L^2\mathfrak{g})\to C^*\bar{G}\to C^*\mathfrak{g}^*\rtimes_{\mathrm{Ad}^*}G\to 0.
\end{equation}
In what follows it will be helpful to think of $C^*\bar{G}$ as an fiber over $\R$ whose fiber over $0$ is equal to $C^*\mathfrak{g}^*\rtimes_{\mathrm{Ad}^*}G$ and whose fiber over $t\neq 0$ is equal to $\mathbb{K}(L^2\mathfrak{g}).$ 

Let $V,W\to M$ be $\Z/2\Z$ graded vector bundles, $$D_1\in \Gamma(\mathcal{U}_k(\mathfrak{g}^*\oplus \R,W_0,W_1)), D_2\in \Gamma(\mathcal{U}_k(\mathfrak{g},V_0,V_1)), $$ be operators that satisfy the Rockland condition with $k\in \N$. Here $\mathfrak{g}^*\oplus \R$ is the commutative Lie algebra. Let $$\psi\in \mathrm{Aut}(\mathfrak{g}^*\oplus \R),\; \psi(l,t)=(l,-t).$$
The operator $\psi(D_1)$ clearly satisfies the Rockland condition. By \cref{main thm on ind crossed product}, there exists $\epsilon>0$ such that for $c\in ]0,\epsilon[$, $D_1\# cD_2$ and $\psi(D_1)\# cD_2$ satisfy the Rockland condition on $\bar{G}=(\mathfrak{g}^*\oplus \R)\rtimes G.$ To simplify the notation, we will suppose that $\epsilon>1, c=1$.

For each $x\in M$, one has the operators $$\pi_x^+(D_1\# D_2),\pi^+_x(\psi(D_1)\# D_2):L^2\mathfrak{g}\otimes (V\otimes W)_0\to L^2\mathfrak{g}\otimes (V\otimes W)_1.$$ \Cref{compact resolvent} implies that each is invertible.

\begin{prop}\label{diff is compact}The operator $\pi_x^+(D_1\# D_2)\pi_x^+(\psi(D_1)\# D_2)^{-1}$ is an invertible bounded and is equal to $1$ modulo compact operators.
\end{prop}
\begin{proof}
To simplify the notation, I will not use $Op$ as in \cref{Sect Symbols}. By \cref{Parametrix symbols}, there exists a symbol $S$ of order $-k$ such that $$S(\psi(D_1)\# D_2)- 1,(\psi(D_1)\# D_2)S-1$$ are of order $-1$. One has \begin{align*}
   \pi^+_x((D_1\#D_2)S)-\pi^+_x(D_1\#D_2) \pi_x^+(\psi(D_1)\# D_2)^{-1}\\=\pi^+\big((D_1\#D_2)(S(\psi(D_1)\# D_2)- 1)\big) \pi_x^+(\psi(D_1)\# D_2)^{-1}
\end{align*}
The first operator is bounded by \cref{symbols bounded}, and the third is compact by \cref{symbols bounded} and \cref{compact resolvent}. It follows that $\pi_x^+(D_1\# D_2)\pi_x^+(\psi(D_1)\# D_2)^{-1}$ is invertible and bounded and is equal to $\pi_x^+((D_1\# D_2)S)$ modulo compact operators.

Let $\phi:\bar{G}\to \mathfrak{g}^*\rtimes_{\mathrm{Ad}^*} G$ be the natural projection. Since $$\phi(D_1\# D_2)=\phi(\psi(D_1)\# D_2),$$ it follows that $$(D_1\# D_2)S-1$$ is a symbol of order $0$ whose pushforward by $\phi$ is of order $-1$. If we denoted by $F$ the Euclidean Fourier transfom on $\bar{G}$ and use $(\xi,y,t)\in \mathfrak{g}^*\times \mathfrak{g}\times \R$ as coordinates for $\bar{g}^*$, then by \cref{Fourier transform}, $$F((D_1\# D_2)S-1)$$ is a function which is homogeneous of degree $0$ at infinity and of degree $-1$ when restricted to $(\xi,y,0).$ It follows from the fundamental theorem of calculus $$F((D_1\# D_2)S-1)(\xi,y,t)=F((D_1\# D_2)S-1)(\xi,y,0)+tL(\xi,y,t),$$ where both $L,F((D_1\# D_2)S-1)(\xi,y,0)$ are symbols of order $-1$.

It is straightforward to see that $\pi^+_x$ applied to a symbol only depends on the restriction of the Euclidean Fourier transform of the symbol to $\mathfrak{g}^*\times \mathfrak{g}\times \{1\}.$ It follows that $$\pi_x^+((D_1\# D_2)S-1)$$ agrees with $\pi^+_x$ of a symbol of order $-1$. The proposition then follows from \cref{symbols bounded} and a classical result of Dixmier which says that irreducible representations of nilpotent groups maps the $C^*$-algebra into compact operators, see \cite[13.11.12]{MR0458185}.
\end{proof}
\begin{rem}\label{rem aux}We actually proved that $(D_1\# D_2)S-1$ is locally (in the $\R$-direction of $C^*\bar{G}$) compact as an operator on $C^*$-modules.
\end{rem}

\Cref{diff is compact} implies that the family of operators $$\pi_x^+(D_1\# D_2)\pi_x^+(\psi(D_1)\# D_2)^{-1}:L^2\mathfrak{g}\otimes (V\otimes W)_1\to L^2\mathfrak{g}\otimes (V\otimes W)_1$$ as $x$ varies defines an element in $K^1(\mathcal{K}(L^2\mathfrak{g})).$ This group is equal to $K^1(M)$ by Morita equivalence.
\begin{theorem}\label{index theorem in sect 3}
The following identity holds
\begin{equation}
\Ind_{\bar{G}}^{-1}([D_1\# D_2])=\mathrm{Th}[\pi_x^+(D_1\# D_2)\pi_x^+(\psi(D_1)\# D_2)^{-1}],
\end{equation}where $\mathrm{Th}:K^1(M)\to K^0(\bar{\mathfrak{g}}^*)$ is the topological Thom isomorphism the spin structure on $\bar{\mathfrak{g}}=\mathfrak{g}\oplus \mathfrak{g}^*\oplus \R$.
\end{theorem}
\begin{proof}
The proof is an adaptation of the proof for the Heisenberg group in \cite{MR2680396}. The theorem follows from  \cref{lem 1} and \cref{lem 2} \qedhere
\begin{lem}\label{lem 1}Let $\partial:K_1(\mathcal{K}(L^2\mathfrak{g}))\to K_0(C^*\bar{G})$ be the map obtained from the composition $$K_1(\mathcal{K}(L^2\mathfrak{g}))\xrightarrow{\mathrm{Suspenstion}}K_0(C_0(]0,+\infty[)\otimes \mathcal{K}(L^2\mathfrak{g}))\xrightarrow{K(\mathrm{inclusion})}K_0(C^*\bar{G}).$$ Then the following identity holds $$\partial([\pi_x^+(D_1\# D_2)\pi_x^+(\psi(D_1)\# D_2)^{-1}])=[D_1\# D_2]$$
\end{lem}
\begin{proof}
Let $\tilde{D_1}$ be the symbol on the commutative group $\mathfrak{g}^*\oplus \R$ which is defined by the formula \begin{align*}
F(\tilde{D}_1)(\xi,t)=F(D_1)(\xi,-|t|)= \begin{cases}F(\psi(D_1))(\xi,t),\quad t>0\\F(D_1)(\xi,t),\quad t\leq 0
\end{cases},\quad \xi \in \mathfrak{g},
\end{align*}
where here $F$ means the Euclidean Fourier transform of differential operators. Since $D_1$ satisfies the Rockland condition, the operator $\tilde{D}_1$ satisfies the Rockland condition as well. Furthermore $[\tilde{D}_1]\in K_0(C^*\mathfrak{g}^*\oplus \R)=K_0(C_0(\mathfrak{g}\oplus \R))$ is equal to $0$ because the map $(\xi,t)\to (\xi,-|t|)$ is nullhomotopic.

By \cref{main thm on ind crossed product}, the operator $\tilde{D}_1\# D_2$ satisfy Rockland condition and its class in $K$-theory is trivial.\footnote{The operator $\tilde{D}_1\# D_2$ does satisfy the Rockland condition without any need for constants. In all cases it doesn't matter one could choose $\epsilon$ small enough so that $\tilde{D}_1\# c D_2$ satisfies the Rockland condition as well.} 
Let $S$ be a parametrix for $\tilde{D_1}\# D_2$. Consider the operator $(D_1\# D_2)S.$ This operator satisfies the Rockland condition, and its class in $K$-theory is equal to the class of $D_1\# D_2.$ The restriction of the  operator $(D_1\# D_2)S$ to $]-\infty,0]$ part of $C^*\bar{G}$ is equal to $\mathrm{Id}$ modulo $C^*$-compact operator. Also $(D_1\# D_2)S$ being a symbol of order $0$, implies that $$\alpha_t((D_1\# D_2)S)-(D_1\# D_2)S$$ is $C^*$-compact for every $t\in \R^+_*.$ The lemma immediately follows from the definition of the suspension map in $K$-theory. See \cite[lemma 6.2.1 and the proof of theorem 6.3.1]{MR3261009}.
\end{proof}
\begin{lem}\label{lem 2}
The following diagram commutes \begin{equation}\label{diagram commut thm sect 3}
 \begin{tikzcd}K^1(\mathcal{K}(L^2\mathfrak{g}))\arrow[r,"\partial"]\arrow[d,"\mathrm{Morita}"]& K_0(C^*\bar{G})\\K^1(M)\arrow[r,"\mathrm{Thom}"]& K^0(\bar{\mathfrak{g}}^*)\arrow[u,"\Ind_{\bar{G}}"]
\end{tikzcd}
\end{equation}
\end{lem}
\begin{proof}
Let $\tilde{G}$ be a bundle of Carnot groups on $M\times [0,1]$ whose Lie algebra at $(x,t)$ is equal to $\bar{\mathfrak{g}}_x$ with the Lie bracket $$[g,h]:=t^{n}\alpha_{t^{-1}}([g,h]).$$

For $t\neq 0$, $\tilde{G}_{x,t}$ is isomorphic to $\bar{G}_x$ by the map $$g\to t^{-n}\alpha_{t^{n}}(g).$$
For $t=0$, $\tilde{G}_{x,0}=\mathfrak{g}\oplus\mathfrak{g}^*\oplus \R$ is the Heisenberg group. It follows from \cref{def index map} that one has an index map $K(C^*\tilde{G}_{|M\times \{0\}})\to K(C^*\bar{G}).$ Commutativity of the diagram \ref{diagram commut thm sect 3} follows then from that the same one for the Heisenberg group. See \cite[proposition 4.6.1]{MR3261009}.
\end{proof}
\end{proof}

\begin{rem}\label{Morita equiv remark}
The Morita equivalence $K^1(\mathcal{K}(L^2\mathfrak{g}))=K^1(M)$ can be described as follows : choose any increasing finite dimensional vector bundles $$0\subseteq L_{1}\subseteq L_2\cdots \subseteq L^2\mathfrak{g}\otimes (V\otimes W)_1=\overline{\cup_{i}L_i}.$$ Let $p_n$ be the orthogonal projection onto $L_n$, then for $n$ big enough $$p_n\left(\pi_x^+(D_1\# D_2)\pi_x^+(\psi(D_1)\# D_2)^{-1}\right)p_n\in \mathrm{Aut}(L_n).$$ The element they define in $K^1(M)$ stabilises for $n$ big enough.
\end{rem}
\begin{rem}All statements in section \ref{Rockl sect}, \ref{Connes-Thom sect}, \ref{index theorem} and \cref{Sect Symbols} without exception can be equally well stated for bundles over locally compact spaces with minor modifications.
\end{rem}
\subsection{The $K_1$ statement}
Let $M$ be a compact smooth manifold, $F_0,F_1:M\to Gl(H)$ two families of invertible self-adjoint operators which are continuous for the strong $*$-topology. If one is given a continuous path $$c(t):M\times [0,1]\to Fred_{sa}(H)$$ connecting $F_0$ and $F_1$, then one immeditetly gets an element in $\KK_1(\C,C(M)\otimes C_0(\R))=K^0(M).$ This element is called the (higher) spectral flow of $c(t).$ See \cite{Dai1996} for an alternative definition which doesn't use $\KK$ theory.

 Let us return to the situation of \cref{index sect 3}, but with the additional hypothesis that $W_0=W_1$ and $D_1$ is symmetric. In this case $D_1\#D_2$ defines an element in $K_1(C^*\bar{G})$. Let $F_0,F_1$ be the bounded transforms of $D_1\#D_2$ and $\psi(D_1)\# D_2$. \Cref{diff is compact} implies $F_0$ and $F_1$ are equal modulo compact operators. We will denote by $sf(F_0,F_1)\in K^0(M)$ the spectral flow between $F_0$ and $F_1$ with the path taken being the affine path.
\begin{theorem}\label{index theorem in sect 3 K0}
	The following identity holds
	\begin{equation}
	\Ind_{\bar{G}}^{-1}([D_1\# D_2])=\mathrm{Th}[sf(F_1,F_0)],
	\end{equation}where $\mathrm{Th}:K^1(M)\to K^0(\bar{\mathfrak{g}}^*)$ is the topological Thom isomorphism the spin structure on $\bar{\mathfrak{g}}=\mathfrak{g}\oplus \mathfrak{g}^*\oplus \R$.
\end{theorem}
\begin{proof}
	The proof is the same as the proof \cref{index theorem in sect 3}. The only dfference is in the proof of \cref{lem 1}. One still defines the operator $\tilde{D}_1$, and one still obtains that the class of $\tilde{D}_1\# D_2$ is $0$. The proof then changes as follows : one computes the sum $$[D_1\# D_2]+sf(F_0,F_1).$$ by representing this sum as a concatinaton of the path of the first and the second element (here $C^*\bar{G}$ is seen as a fiber over $\R$), one obtains by using an affine homotopy and \cref{rem aux} that $$[D_1\# D_2]+sf(F_0,F_1)=[\tilde{D}_1\# D_2]=0.$$ The rest of the proof is then the same.
\end{proof}
We will write $sf[\psi(D_1\# D_2),D_1\# D_2]$ to mean the spectral flow of their bounded transforms.
\section{Carnot Pseudo-Differential calculus}\label{Sec Carnot diff calc}
Let $M$ be a closed manifold, $0=H^0\subseteq H^1\subseteq\cdots\subseteq H^r\subseteq TM$ vector subbundles with the property that \begin{equation}\label{comm relat}
 [\Gamma(H^i),\Gamma(H^j)]\in \Gamma(H^{i+j}),
\end{equation} where $H^k:=TM$ for $k>r$.
\paragraph{Inhomogeneous tangent bundle}Let $x\in M, \;a\in H^i_x,\;b\in H^j_x$. If $X\in \Gamma(H^i)$, $Y\in \Gamma(H^j)$ such that $X(x)=a$, $Y(x)=b$, then the value $[X,Y](x)\mod H^{i+j-1}_x$ doesn't depend on the choice of $X,Y$ and it vanishes if either $a\in H^{i-1}$ or $b\in H^{j-1}.$ Hence one gets a bundle of bilinear maps $$[\cdot,\cdot]:\frac{H^i}{H^{i-1}}\times \frac{H^j}{H^{j-1}}\to \frac{H^{i+j}}{H^{i+j-1}}$$

Since Lie bracket of vector fields satisfies Jacobi identity, it follows that those bilinear maps define for any $x\in M$ a Carnot Lie algebra structure on $$\mathfrak{g}_x:=\oplus_{k=1}^{k=r} H^{k}_x/H^{k-1}_x.$$ The simply connected group integrating this Lie algebra will be denoted by $G_x$ (by the Baker-Campbell-Hausdorff formula as a space $G_x=\mathfrak{g}_x$). The groups $G_x$ glue together to form a bundle of Carnot Lie groups on $M$ that we denote by $G$.

\paragraph{Filtration on differential operators}Let $E\to M$ be a vector bundle, $\nabla$ a connection on $E$. We equip $\mathrm{DO}(M,E),$ the algebra of differential operators acting on $E$ with the smallest increasing filtration $\mathcal{F}^i$ such that \begin{enumerate}
\item $\mathcal{F}^{-1}=0$, $\mathcal{F}^0=\Gamma(\End(E))$
\item if $X\in \Gamma(H^i),$ then $\nabla_{X}\in \mathcal{F}^i$
\item $\mathcal{F}^i\mathcal{F}^j\subseteq \mathcal{F}^{i+j}.$
\end{enumerate}
A differential operator belongs to $\mathcal{F}^n$ if locally $D$ can be written as sum of operators of the form $L\nabla_{X_1}\nabla_{X_2}\cdots\nabla_{X_k}$ with $L\in \Gamma(\End(E))$, $X_i\in \Gamma(H^{w(i)})$ and $\sum_{i=1}^kw(i)\leq n.$ It follows that the filtration is independent of $\nabla.$
\begin{prop}The following is an isomorphisms of graded algebras \begin{align*}
\oplus_{k\geq 0}\mathcal{F}^{k}/\mathcal{F}^{k-1}&\to \oplus_{k\geq 0}\Gamma(\mathcal{U}_k(\mathfrak{g},E))\\
\nabla_{X}&\to (x\to [X(x)\mod H^{i-1}_x])\quad \text{for} \;X\in \Gamma(H^i)\\
L&\to L\quad \text{for} \;L\in \Gamma(\End(E))
\end{align*}
\end{prop}
\begin{proof}
Surjectivity is clear from the definition of the universal enveloping algebra.  Suppose that $D\in \mathcal{F}^n$ and its image is $0$. It is enough to prove that $D$ is locally in $\mathcal{F}^{n-1}.$ Suppose that $D$ is the sum of operators of the form $L\nabla_{X_1}\nabla_{X_2}\cdots\nabla_{X_k}$. The commutator relation \cref{comm relat}, allows one to reorder the $\nabla_{X_i}'s$ as one wishes while staying in $\mathcal{F}^n$. Injectivity follows then from the Poincaré–Birkhoff–Witt theorem on basis of enveloping Lie algebras.
\end{proof}

If $D\in\mathcal{F}^i$, then we will denote by $\sigma_i(D)$ its image in $\Gamma(\mathcal{U}_i(\mathfrak{g},E))$. We say that $D$ satisfies the Rockland's condition if $\sigma_i(D)$ satisfies the Rockland condition at every point $x\in M$.
\begin{theorem}If $D$ is a differential operator that satisfies the Rockland condition, then $D(1+D^*D)^{-\frac{1}{2}}$ is a Fredholm operator.
\end{theorem}
\begin{proof}
This was proved in an unpublished manuscript by Melin \cite{MelinUnpubl}. A proof can also be found in \cite[theorem 3.13]{Dave:aa}
\end{proof}
\paragraph{Deformation groupoid}For the index theory we need the deformation Lie groupoid \begin{align*}
 \mathbb{T}_HM= M\times M\times ]0,1]\sqcup G\times {0}\rightrightarrows M\times [0,1]
\end{align*}
with the structure maps given by \begin{align*}
r(x,y,t)=(x,t),\; r(x,g,0)=(x,0)\\ s(x,y,t)=(y,t),\; s(x,g,0)=(x,0)\\
(x,y,t)\cdot (y,z,t)=(x,z,t)\quad \text{for }\; t\neq 0\\
(x,g,0)\cdot (x,h,0)=(x,gh,0)\quad \text{for }\; g,h\in G_x\\
(x,y,t)^{-1}=(y,x,t),\; (x,g,0)^{-1}=(x,g^{-1},0)
\end{align*}

The construction of this Lie groupoid can be found in \cite{Erp:2016aa,Choi:2015aa,Mohsen:2018aa}.
\paragraph{Index map} Following \cref{def index map} applied to $\mathbb{T}_HM$, one obtains an index map that will be denoted by $\Ind_{CAS}$ for Carnot-Atiyah-Singer $$K_0\left(C^*G\right)\xrightarrow{\ev_0^{-1}} K_0(C^*(\mathbb{T}_H(M)))\xrightarrow{\ev_1}K_0(\mathcal{K}(L^2(M)))=\Z.$$ 

\begin{theorem} \label{diagram commutes double def}The following diagram commutes

$$\begin{tikzcd}
K_0(C^*G)\arrow[r,"\Ind_{CAS}"]& \Z\\
K_0(C_0(\mathfrak{g}^*))\arrow[u,"\Ind_G"]\arrow[r,"S"]&K_0(C_0(T^*M))\arrow[u,"\Ind_{AS}"']
\end{tikzcd},$$ where $\Ind_{AS}$ is the Atiyah-Singer index map, and $S:TM\to \mathfrak{g}$ any linear isomorphism induced from a splitting $L^i\oplus H^i=H^{i+1}$. The map $S$ is unique up to homotopy.
\end{theorem}
\begin{proof}
In \cite{Mohsen:2018aa}, the deformation groupoid is constructed as a restriction of a bigger groupoid  which is given by $$M\times M\times ]0,1]^2\sqcup G\times \{0\}\times ]0,1]\sqcup TM\times ]0,1]\times \{0\}\sqcup \mathfrak{g}\times \{0\}\times\{0\}\rightrightarrows M\times [0,1]^2.$$ The restriction of this groupoid to $[0,1]\times \{1\}$ gives the deformation groupoid $\mathbb{T}_HM$, and the index map associated is $\Ind_{CAS}$.

The restriction to $\{1\}\times [0,1]$ gives Connes tangent groupoid. It is proved in \cite[lemma 6 on page 109]{MR1303779} that the index map associated to Connes tangent groupoid is the Atiyah-Singer index map. 

The restriction to $\{0\}\times [0,1]$ gives the deformation groupoid constructed in \cref{Connes-Thom sect} that was used to construct the map $\Ind_G$. 

The restriction to $[0,1]\times \{0\}$ gives the deformation of $TM$ onto the graded vector bundle $\mathfrak{g}$. The index map associated is then the map $S$.

The proof of the commutativity of the diagram is then the same as that of \cref{commt diagram groups}.
\end{proof}
\begin{theorem}\label{thm index = analytic index}
 Let $D\in \Gamma(\mathcal{U}_k(\mathfrak{g},V_0,V_1))$ be a differential operator satisfying the Rockland condition. Then $$\Ind_{CAS}([\sigma(D)])=\dim(\ker(D))-\dim(\ker(D^t)).$$
\end{theorem}
\begin{proof}
The construction of the inhomogeneous deformation groupoid makes it so that the operator $\mathbb{D}:\Gamma_c(\mathbb{T}_HM,r^*\pi_M^*V_0)\to \Gamma_c(\mathbb{T}_HM,r^*\pi_M^*V_0)$ given by the following formula is smooth \begin{align*}
\mathbb{D}f(x,y,t)=t^kDf(x,y,t)\\
\mathbb{D}f(x,g,0)=\sigma_k(D)(x)f(x,g,0)
\end{align*}
where in the first formula $D$ acts on the $x$-variable, and the in the second $\sigma_k(D)(x)$ is the symbol of $D$ at $x$. Here $\pi_M\circ r:T_HM\to M$ is the composition of the range map with the projection $\pi_M:M\times [0,1]\to M$
\begin{lem}\label{lem rockland operator manifold}The operator $\mathbb{D}$ is regular and the multiplier $$\mathbb{D}(1+\mathbb{D}^*\mathbb{D})^{-\frac{1}{2}}\in \mathcal{L}(\Gamma(V_0)\otimes_{C(M)}C^*\mathbb{T}_HM,\Gamma(V_1)\otimes_{C(M)}C^*\mathbb{T}_HM)$$ is Fredholm.
\end{lem}
\begin{proof}
This follows from the existence of a parametrix for $D$, see \cite[section 9]{Erp:2015aa}.
\end{proof}
It follows from \cref{lem rockland operator manifold} and the definition of $\Ind_{CAS}$ that $\Ind_{CAS}([\sigma(D)])$ is equal to the class of $D(1+D^*D)^{-\frac{1}{2}}$ in $K^0(\mathcal{K}(L^2M))=\Z.$
\end{proof}

\paragraph{Index theorem} Let $M$ be a spin$^c$ manifold, $D:\Gamma(V_0)\to 
\Gamma(V_1)$ an operator satisfying Rockland condition on $M$, 
$$\Ind_a(D):=\dim(\ker(D))-\dim(\ker(D^t)).$$ \Cref{thm index = analytic 
index}, \cref{diagram commutes double def}, the Atiyah-Singer index 
theorem imply $$\Ind_a(D)=\int_M 
\ch(\mathrm{Th}(\Ind_G^{-1}([\sigma(D)])))\mathrm{Td}(TM).$$

Here we use the Todd class for spin$^c$ bundles defined in \cite{BDauglas}. Let $\slashed{D}\in \Gamma(\mathcal{U}_k(\mathfrak{g}^*\oplus 
\R,W_0,W_1))$ be the Dirac operator constructed in \cref{examps sect 
1}.2. Then the operator $\slashed{D}\#c\sigma(D)$ acting on 
$\bar{G}=(\mathfrak{g}^*\oplus \R)\rtimes G$  for $c>0$ small enough 
satisfies Rockland condition. Furthermore 
$$\Ind_{\bar{G}}^{-1}([\slashed{D}\#c\sigma(D)])=[\slashed{D}]\otimes 
\Ind_{G}^{-1}([\sigma(D)]).$$
It follows that 
$$\Ind_a(D)=\int_M\ch(\mathrm{Th}(\Ind_{\bar{G}}^{-1}([\slashed{D}\#c\sigma(D)])))\mathrm{Td}(TM).$$
\Cref{index theorem in sect 3} implies
\begin{theorem}\label{index formula thm} there exists $\epsilon>0$, such that for $0<c<\epsilon$,
 \begin{enumerate}
\item if $M$ is odd dimenional then \begin{align*}
 \Ind_a(D)=\int_M\ch[\pi_x^+(\slashed{D}\#c\sigma(D))\pi_x^+(\psi(\slashed{D})\#c\sigma(D))^{-1}]\mathrm{Td}(TM)
\end{align*}
\item if $M$ is even dimentional, then 
 \begin{align*}
\Ind_a(D)=\int_M\ch[sf(\pi_x^+(\psi(\slashed{D})\#c\sigma(D)),\pi_x^+(\slashed{D}\#c\sigma(D))]\mathrm{Td}(TM)
\end{align*}
\end{enumerate}
where $\pi^+$ is the left regular representation of $G$ and $\psi$ is as in \cref{sect Index thm}, the map $\psi(\xi,t)=(\xi,-t)$.
\end{theorem}
If the manifold is only oriented, then one can replace $\slashed{D}$ 
with the signature operator.

\appendix

\section{Symbols of psuedo-differential operators}\label{Sect Symbols}
The definition and the basic properties of the symbol calculus for Carnot groups can be found in \cite{MR764508,MR739894,MR1185817,Erp:2015aa}. We need as well extensions to the above for bundles of Carnot Lie groups. This is done in \cite{MR1185817} for locally trivial bundles, in \cite{MR2417549} for bundles of $2$-rank groups, in \cite{Dave:aa,Dave:ab} for arbitrary bundles..

Other than the references given above, we will present a way (\cref{lift Rock}) that allows one to reduce some of the results of this section to the case where the bundle is a constant bundle. The results then follow from their counterpart in \cite{MR1185817}.
\begin{dfn} Let $G\to X$ be a bundle of Carnot Lie group on a locally compact space $X$. The space $S^m(G)$ denotes the space of distributions $u\in \mathcal{E}'(G)$ such that 
\begin{enumerate}
\item $u$ is transversal to the map $\pi:G\to X$ (\cite{Androulidakis:2009aa,MR3669118}).  This means that if $\pi:G\to X$ denotes the projection map, then $\pi_*(u)\in C^\infty(X).$
\item  for every $\lambda\in \R^+$, \begin{equation}\label{homog condition symbols}
 \alpha_\lambda^*u-\lambda^m u\in C^\infty_c(G), 
\end{equation}where $(\alpha_\lambda^*u,f):=(u,f(\alpha_\lambda(\cdot)))$ for $f\in C^\infty(G).$
\end{enumerate}

\end{dfn}
The first condition says that $u$ is a 'smooth' family of distributions $u_x\in \mathcal{E}'(G_x).$ Since $u$ has compact support and for any $a\in G\backslash X$, $\lim_{\lambda\to +\infty}\alpha_\lambda(a)=+\infty$. \Cref{homog condition symbols} implies that $u$ is smooth on $G\backslash X.$
We denote by $F:\mathcal{E}'(G)\to  \mathcal{S}'(\mathfrak{g}^*)$ the fiber wise Euclidean Fourier transform, when $G$ is identified with $\mathfrak{g}$.
\begin{prop}\label{Fourier transform}If $u\in S^m(G)$, then $Fu\in S'(\mathfrak{g}^*)$ is smooth on $\mathfrak{g}^*\backslash \{0\}$ and outside a compact neighbourhood of $\{0\}$, the function $Fu$ is a Schwartz function plus a homogeneous function of degree $\alpha$.
\end{prop}
\begin{proof}
If $X$ is a single point, then this is \cite[proposition 1.6,1.7]{MR764508}. The same proof generalises to the bundles.
\end{proof}
If $u_1\in S^{m_1}(G)$ and $u_2\in S^{m_2}(G)$, then their convolution is defined by $$(u_1\star u_2,f)=(u_1,x\to (u_2,y\to f(y x))).$$ \begin{prop}$u_1\star u_2\in S^{m_1+m_2}(G).$
\end{prop}
Let $u\in S^m(G)$. For each $x\in X$, one defines the operator \begin{align*}
 Op(u_x):C^\infty_c(G_x)\to C^\infty_c(G_x)\\ f\to \bigg(a\to \big(u_x,b\to f(ba)\big)\bigg),\quad f\in C^\infty_c(G),a,b\in G_x.
\end{align*} Smoothness of the family $u_x$, implies that $Op(u_x)$ glue together to form an operator $$Op(u):C^\infty_c(G)\to C^\infty_c(G).$$

The following identity follows directly from the definition $$Op(u)(f\star g)=Op(u)(f)\star g,\; Op(u_1\star u_2)=Op(u_1)\circ Op(u_2).$$ 

If $D\in  \mathcal{U}_k(G)$, then the distribution $(u,f)=Df(e)$ belongs to $S^k(G).$ Since $D$ is right-invariant, $Op(u)=D$. This justifies considering $\mathcal{U}_k(G)\subseteq S^k(G).$

Let $u^*$ be the distribution defined by $(u^*,f)=\overline{(u,f^*)}$. It is straightforward to see that $u^*\in S^m(G)$ and $Op(u^*)$ is contained in the adjoint of $Op(u)$.

If $\phi:G_1\to G_2$ is a bundle of grade preserving group homorphisms between Carnot Lie groups, then pushforward of distributions defines a algebra homomorphism $\phi_*:S^*(G_1)\to S^*(G_2)$.
\begin{lem}\label{surjectivity symbols}If $\phi$ is a submersion, then $\phi_*:S^m(G_1)\to S^m(G_2)$ is surjective.
\end{lem}
\begin{proof}
Let $f\in C^\infty_c(G_1)$ such that $f=1$ on a neighbourhood of $0$ and $\phi_*(f)=1\in C^\infty_c(G_2)$. If $u\in S^m(G_2)$, then $f\phi^*u\in S^m(G_1)$ and its image by $\phi_*$ is equal to $u$.
\end{proof}
\begin{dfn}A symbol $u$ is said to satisfy Rockland condition at a point $x\in X$,  if there exists a compact $K\subseteq \hat{G}_x$ of the dual space, such that for every $\pi\notin K$, $\pi(Op(u))$ and $\pi(Op(u)^*)$ are injective. A symbol satisfy the Rockland condition if it does at every point $x\in X$.
\end{dfn}
For differential operators Rockland condition is the same as the one defined in \cref{Rockl sect}, because differential operators are equivariant, not just equivariant up to smooth functions.
\begin{theorem}\label{Parametrix symbols}If $u\in S^{k}(G)$ is a symbol, then the following are equivalent \begin{enumerate}
\item $u$ satisfies Rockland condition
\item there exists $l,r\in S^{-k}(G)$ such that  $1-l\star u,1-u\star r\in S^{-1}(G).$
\end{enumerate}
\end{theorem}
\begin{proof}
If the bundle is constant, this is \cite[Theorem 2.5]{MR1185817}. A proof for the general case is given in \cite{Dave:aa,Dave:ab}. A different proof is as follows : given $n\in \N^+$ and $a_1,\cdots,a_n\in \N$, one can construct the free $n$-step Carnot group generated by variables $X^1_1,\cdots,X^1_{a_1},\cdots,X^n_1,\dots,X^n_{a_n}$ where $X^i_j$ has degree $i$. By compactness of $X$, one can choose a finite number of generators $X^i_j\in \Gamma(\mathfrak{g}_i)$. Let $\mathbb{G}$ be the constant bundle over $X$ of free Carnot groups with variables $X_j^i$. By construction one has a submersion $\phi:\mathbb{G}\to G.$ The theorem follows from \cref{lift Rock} applied to $u$ and $u^*$ together with the results of \cite{MR1185817}.\qedhere
\begin{lem}\label{lift Rock}Let $\phi:G_1\to G_2$ be a submersion. If $u\in S^m(G_2)$ satisfies the Rockland condition, then there exists $v\in S^{2m}(G_1)$ such that $\phi_*(v)=u^*\star u.$
\end{lem}
\begin{proof}
Let $\tilde{u}$ be any lift of $u$ to $S^m(G_1)$ and $k\in S^{m}(\ker (\phi))$ an operator that satisfy the Rockland condition (for the bundle $\ker(\phi)$), then $\tilde{u}^*\star\tilde{u}+k^*\star k$ is a lift of $u$ which satisfies the Rockland condition.
\end{proof}
\end{proof}
\begin{theorem}\label{symbols bounded}\begin{enumerate}
\item If $m\leq 0$, then $\overline{Op(u)}\in \mathcal{L}(C^*G).$
\item If $m<0$, then $\overline{Op(u)}\in \mathcal{K}(C^*G)=C^*G$.
\item If $m>0$, and $u$ satisfy the Rockland condition, then $\overline{Op(u)}$ is a regular operator and $\overline{Op(u)}^*=\overline{Op(u^*)}.$
\end{enumerate}
\end{theorem}
\begin{proof}
 By \cite[Theorem 5.3]{MR739894}, the operators $Op(u)$ and $Op(u^*)$ are bounded. For $m<0$, then it is enough to show that $Op(u^*\star u)\in S^{2m}(G)$ is compact. For $k$ big enough then the Euclidean Fourier transform of $Op(u^*\star u)^k$ is integrable (see \cite[proposition 1.9]{MR764508}), hence in $C^*G$. It follows that $Op(u^*\star u)$ is compact as well. For $m>0$, the proof is the same as the proof for classical pseudo-differential operators given \cite[proposition 3.6.2]{MR2227132}.
\end{proof}
\begin{prop}Let $u\in S^k(G)$ be a symbol. The set of all $x\in M$ such that $u$ satisfy the Rockland condition at $x$ is an open set.
\end{prop}

\begin{proof}This is a direct corollary of \cref{lift Rock} applied to $D$ and $D^t$ together with \cite[Theorem 2.5]{MR1185817}
\end{proof}

\begin{prop}
Let $u\in S^k(G)$ be a symbol that satisfies Rockland condition with $k>0$, then
$$Op(u)(1+Op(u^*)Op(u))^{-\frac{1}{2}}$$ is a Fredholm operator in the sense of Kasparov $C^*$-modules.
\end{prop}
\begin{proof}This is a direct consequence of \cref{symbols bounded} and \cref{Parametrix symbols}.
\end{proof}
Let $V\to X$ be a Hermitian vector space. Clearly all the above generalises to $\End(V)$ valued symbols.

\begin{prop}\label{compact resolvent}
 Let $D\in \Gamma_k(\mathcal{U}(\mathfrak{g},V))$ be a differential operator satisfying the Rockland condition with $k>0$. It follows that $\left(1+\overline{D}^*\overline{D}\right)^{-1}\in C^*G.$ Hence for every $x\in X$, $\pi$ a non trivial unitary irreducible representation of $G_x$, $\pi(\overline{D})$ is invertible and its inverse is compact
\end{prop}
\begin{proof}
Let $L\in S^{-2k}(G)$ be such that $D^t\star D\star L -1=R\in S^{-1}(G).$ It follows that\begin{align*}
 (1+\overline{D}^*\overline{D})^{-1}=\overline{Op(L)}- (1+\overline{D}^*\overline{D})^{-1} \overline{Op(R)}\\-(1+\overline{D}^*\overline{D})^{-1}\overline{Op(L)}. 
\end{align*}
The first part follows from \cref{symbols bounded}. The Rockland condition says that $\pi(\overline{D})$ is injective and has dense image. It suffices then to show that $\pi(\overline{D}(1+\overline{D^tD})^{-\frac{1}{2}})$ is Fredholm. This is a  direct consequence of the first part of \ref{compact resolvent} and a well known result of Dixmier that the $C^*$-algebra of nilpotent Lie groups is mapped into compact operators by irreducible representations, see \cite[13.11.12]{MR0458185}. The compactness of $\pi(\overline{D})^{-1}$ is obvious.
\end{proof}
 \bibliography{../Biblatex}
\bibliographystyle{alpha}

\end{document}